\documentclass{amsart}
\usepackage{psfrag}
\usepackage{epsfig}
\usepackage{graphicx}
\usepackage{amsmath}
\usepackage{amssymb}
\usepackage{amsthm}
\usepackage{subfigure}
\usepackage{cite}
\usepackage[latin1]{inputenc}
\usepackage{euscript}

\newtheorem{example}{Example}
\newtheorem{proposition}{Proposition}
\newtheorem{lemma}{Lemma}
\newtheorem{theorem}{Theorem}
\newtheorem{remark}{Remark}
\newtheorem{corollary}{Corollary}

\def\Aff{{\mathbb A}}
\def\PP{{\mathbb P}}
\def\Plm{{\mathbb P}^{\ell m-1}}
\def\F{{\mathbb F}}
\def\Fq{{\mathbb F}_q}

\def\Mat{{\mathbb M}}
\def\Matlm{{\Mat}_{\ell\times m}}
\def\Det{{\mathcal D}}
\def\Dt{{\Det}_t}
\def\Ej{{\mathcal E}_j}
\def\Ez{{\mathcal E}_0}
\def\Dtp{\widehat{\Det}_t}
\def\hC{\widehat{C}}
\def\It{{\mathcal{I}}_{t+1}}
\def\Ct{C_{\det}(t;\ell,m)}
\def\Cpt{\widehat{C}_{\det}(t;\ell,m)}
\def\Cone{C_{\det}(1;\ell,m)}
\def\Cpone{\widehat{C}_{\det}(1;\ell,m)}
\newcommand{\wH}{{\mathrm{w_H}}}
\newcommand{\rk}{\operatorname{rank}}

\newcommand{\EuD}{{\EuScript D}}
\newcommand{\EuE}{{\EuScript E}}
\newcommand{\EuF}{{\EuScript F}}
\newcommand{\EuX}{{\EuScript X}}

\newcommand{\GL}{\operatorname{GL}}
\newcommand{\Ev}{\operatorname{Ev}}

\begin{document}
\vspace*{5mm}

\noindent
%%--The title should be placed here
\textbf{\LARGE Linear Codes associated to Determinantal Varieties}
%%--If you do not use any support, just comment the next line or delete it.
\footnote{P. Beelen gratefully acknowledges the support from the Danish National Research Foundation and the National Science Foundation of China (Grant No.11061130539) for the Danish-Chinese Center for Applications of Algebraic Geometry in Coding Theory and Cryptography. \\ \indent
S. R. Ghorpade gratefully acknowledges the support from the Indo-Russian project INT/RFBR/P-114 from the Department of Science \& Technology, Govt. of India and the  IRCC Award grant 12IRAWD009 from IIT Bombay.}

\date{}

\vspace*{10mm}
\noindent
%%--The name(s) of the author(s), their e-mails and addresses should be placed here
\textsc{Peter Beelen} \hfill \texttt{pabe@dtu.dk} \\
{\small Department of Applied Mathematics and Computer Science,  \newline 
Technical University of Denmark, DK 2800, Kgs. Lyngby, Denmark} \\ \\
\textsc{Sudhir R. Ghorpade} \hfill \texttt{srg@math.iitb.ac.in} \\
{\small Department of Mathematics,
Indian Institute of Technology Bombay, \newline
Powai, Mumbai 400076, India} \\ \\
\textsc{Sartaj Ul Hasan} \hfill \texttt{sartajulhasan@gmail.com} \\
{\small Scientific Analysis Group, Defence Research and Development Organisation,  \newline
Metcalfe House, Delhi 110054, India}

\bigskip

\vspace*{5mm}

\begin{center}
\parbox{11,8cm}{\footnotesize
%%--The abstract goes here.
\textbf{Abstract.}
We consider a class of linear codes associated to projective algebraic varieties defined by the vanishing of minors of a fixed size of a generic matrix. 
It is seen that the resulting code has only a small number of distinct weights.
%After establishing some basic results, 
The case of varieties defined by the vanishing of $2 \times 2$ minors is considered in some detail. 
 Here we obtain the complete weight distribution.    
Moreover, several generalized Hamming weights are determined explicitly and it is shown that the first few of them coincide with the distinct nonzero weights. 
%The length and dimension are described explicitly. The minimum distance as well as the complete weight distribution is obtained in the  special case of $2 \times 2$ minors, whereas a bound for the minimum distance is described in another special case. 
One of the tools used is to determine the maximum possible number of matrices of rank~$1$ in a linear space of matrices of a  given dimension over a finite field. In particular, we determine the structure and the maximum possible dimension of linear spaces of matrices in which every nonzero matrix has rank~$1$.
}
\end{center}

%\keywords{linear codes, determinantal varieties, generalized Hamming weight, weight distribution}

\section{Introduction}
A useful and interesting way to construct a linear code is to consider a projective algebraic variety $V$ defined over the finite field $\Fq$ with $q$ elements together with a nondegenerate  embedding in a projective space, and to look at the projective system (in the sense of Tsfasman and Vl\u{a}du\c{t} \cite{TV1}) associated to the $\Fq$-rational points of $V$. A good illustration is provided by the case of Grassmann codes and Schubert codes, which have been of much interest; see, for example, \cite{N}, \cite{GT}, \cite{HJR2}, \cite{X} or the survey \cite{Little}.
 %This has given rise to interesting classes of codes
In this paper we consider a class of linear codes that are associated to classical determinantal varieties. These will be referred to as determinantal codes. The length and dimension of these codes are easy to determine and also one can readily show that they are nondegenerate. We shall then focus on the question of determining the minimum distance and more generally, the complete weight distribution, and also the generalized Hamming weights of determinantal codes. From a geometric viewpoint, this corresponds to determining the %maximum
number of $\Fq$-rational points in all possible hyperplane sections and also in maximal linear sections of determinantal varieties. We give a general description of all the weights of determinantal codes and then analyze in greater details the codes associated to the variety defined by the vanishing of all $2\times 2$ minors of a generic $\ell \times m$ matrix. It is seen in this case that the codes exhibit a curious phenomenon that there are exactly $\ell$ nonzero weights and these coincide with the first $\ell$ generalized Hamming weights which happen to meet the Griesmer-Wei bound. This phenomenon is exhibited by $[n,k]_q$-MDS codes (for instance, the Reed-Solomon codes),  which have exactly $k$ nonzero weights and $k$ generalized Hamming weights given by %$n-k+1,  \dots , n$.  
$n-k+1, n-k+2, \dots , n$. 
Another trivial example is that of the simplex code (i.e., the dual of Hamming code) which has only one nonzero weight and it evidently coincides the first generalized Hamming weight. However, we do not know any other nontrivial examples and determinantal codes appear to be intersting in this regard. Unlike simplex codes, determining \emph{all} generalized Hamming weights of determinantal codes seems difficult. but we make some partial progress here. 

It turns out (although we were not initially aware of it) that codes analogous to determinantal codes were considered in a different context by  Camion \cite{C} and Delsarte \cite{D} who consider codes derived from bilinear forms. In effect, Delsarte obtains the weight distribution of these codes using an explicit determination of the characters of the Schur ring of an association scheme corresponding to these bilinear forms (see end of Section \ref{sec:wts} below for more details). Our approach, however, is entirely different and may be of some interest. Further, results concerning generalized Hamming weights appear to be new. 
The auxiliary results used in finding the generalized Hamming weights %here that 
were alluded to in the abstract, and these (namely, Corollary \ref{cor:maxrank1} and Lemma \ref{maxrank1s})  
may also be of some independent interest.

This work has been presented at the Fourteenth International Workshop on Algebraic and Combinatorial 
Coding Theory (ACCT-XIV) held at Kaliningrad, Russia during September 2014, and an extended 
abstract containing %the 
statements of results appears in the informal proceedings of ACCT-XIV.

\section{Preliminaries}
\label{sec:prelim}

Fix throughout this paper a prime power $q$, positive integers $t, \ell,m$, and a $\ell\times m$ matrix $X = (X_{ij})$ whose entries are independent indeterminates over $\Fq$. We will denote by $\Fq[X]$ the polynomial ring in the $\ell m$ variables $X_{ij}$ ($1\le i \le \ell$, $1\le j \le m$) with coefficients in $\Fq$. As usual, by a \emph{minor} of size $t$ or a $t\times t$ minor of $X$ we mean the determinant of a $t\times t$ submatrix of $X$, where $t$ is a nonnegative integer $\le \min\{\ell, m\}$. As per standard conventions, the only  $0\times 0$ minor of $X$ is $1$. We will be mostly interested in the class of minors of a fixed size, and this class is unchanged if $X$ is replaced by its transpose. With this in view, we shall always assume, without loss of generality, that $\ell\le m$. Given a field $\F$, we denote by $\Matlm (\F)$ the set of all $\ell\times m$ matrices with entries in $\F$. Often $\F=\Fq$ and in this case we may simply write
$\Matlm$ for $\Matlm (\Fq)$. Note that $\Matlm $ can be viewed as an affine space $\Aff^{\ell m}$ over $\Fq$ of dimension $\ell m$. For $0\le t \le \ell$, the corresponding classical determinantal variety (over $\Fq$) is denoted by $\Dt$ and defined as  the affine algebraic variety in $\Aff^{\ell m}$ given by the vanishing of all $(t+1) \times (t+1)$ minors of $X$; in other words
$$
\Dt = \left\{ M\in \Matlm (\Fq) : \rk (M) \le t \right\}.
$$
The affine variety $\Dt$ is, in fact, a cone; in other words, the vanishing ideal $\It$ (which is precisely the ideal of $\Fq[X]$ generated by all $(t+1) \times (t+1)$ minors of $X$) is a homogeneous ideal.
Also it is a classical (and nontrivial) fact that $\It$ is a prime ideal (see, e.g., \cite{survey}).
Thus $\Dt$ can also be
viewed as a projective algebraic variety in $\Plm$, and viewed this way, we will denote it by $\Dtp$. We remark that the dimension of $\Dtp$ is $t(\ell + m -t) -1$ (cf. \cite{survey}). Briefly put, the determinantal code $\Cpt$ is the linear code corresponding to the projective system
$\Dtp \hookrightarrow \Plm (\Fq) = \PP(\Matlm)$. An essentially equivalent way to obtain this code is to consider the  image $\Ct$ of the evaluation map
\begin{equation}
\label{EvMap}
\Ev: \Fq[X]_1 \to \Fq^n \quad \text{defined by} \quad \Ev(f) = c_f:=  \left(f(M_1), \dots , f(M_n) \right),
\end{equation}
where $\Fq[X]_1$ denotes the space of homogeneous polynomials in $\Fq[X]$ of degree $1$ together with the zero polynomial, and
$M_1, \dots , M_n$ is an ordering of $\Dt$.

Recall that in general for a linear code $C$ of length $n$,
% and dimension $k$, i.e., for a $k$-dimensional  subspace $C$ of $\Fq^n$,
 i.e., for a linear subspace $C$ of $\Fq^n$,
the \emph{Hamming weight} of a codeword $c=(c_1, \dots , c_n)$, denoted $\wH(c)$, and the \emph{support weight} of %a subcode $D$ of $C$,
any $D\subseteq C$, denoted $\Vert D\Vert$, are defined by
$$
\wH(c) : =  |\{i: c_i\ne 0\}| \quad \text{and} \quad
\Vert D \Vert  :=  |\{ i: \text{there exists } c\in D \text{ with } c_i \ne0\}|,
$$
where for a finite set $S$, by $|S|$ we denote the cardinality of $S$. The minimum distance of $C$, denoted $d(C)$, and more generally, the $r^{\rm th}$ \emph{higher weight} or the $r^{\rm th}$ \emph{generalized Hamming weight} of $C$, denoted $d_r(C)$, are defined by
\begin{eqnarray*}
d(C) &:= & \min\{\wH(c) : c\in C, \; c\ne 0\} \quad \text{and  for $r=1, \dots , k$,} \quad \\
d_r(C)&:=& \min\{\Vert D \Vert : D \text{ is a subcode of }C \text{ with } \dim D=r\}.
\end{eqnarray*}

The parameters of $\Ct$ determine those of $\Cpt$ and vice-versa. More precisely, we have the following.

\begin{proposition}
\label{AffProj}
Write $C = \Ct$ and $\hC = \Cpt$. Let $n, \,  k, \, d$, and $A_i$ (resp. $\hat{n}, \, \hat{k}, \, \hat{d}$, and $\hat{A}_i$) denote, respectively, the length, dimension, minimum distance and the number of codewords of weight $i$ of $C$ (resp. $\hC$). Then %$\Ct$ (resp. $\Cpt$). Then
$$
n = 1 +\hat{n}(q-1), \quad k= \hat{k}, \quad d = \hat{d}(q-1), \quad \text{and} \quad A_{i(q-1)} = \hat{A}_i \  \text{for $0\le i \le \hat{n}$}.
$$
%for $0\le i \le \hat{n}$.
Moreover $A_n=0$ and more generally, $A_j = 0$ for $0\le j\le n$ such that
%whenever $0\le j\le n$  and
$(q-1) \nmid j$.
%$j\in\{0,1,\dots , n\}$ is not divisible by $(q-1)$.
Furthermore,  if for $1\le r\le k$, we denote by $d_r$ and $A^{(r)}_i$ (resp: $\hat{d}_r$ and $\hat{A}^{(r)}_i$) the $r^{\rm th}$ \emph{higher weight} and the number of $r$-dimensional subcodes of support weight $i$ of $C$ (resp. $\hC$), %$\Ct$ (resp. $\Cpt$),
then $d_r = (q-1)\hat{d}_r$ and $A^{(r)}_{i(q-1)} = \hat{A}_i^{(r)}$ for $0\le i \le \hat{n}$.
\end{proposition}

\begin{proof}
%Observe that
Clearly, $\rk(\lambda M) = \rk(M)$ and $f(\lambda M) =\lambda f(M)$ 
for any \hbox{$M\in \Matlm$}, $f\in \Fq[X]_1$  and $\lambda\in \Fq$ with $\lambda \ne 0$. 
%for any $\lambda\in \Fq$ with $\lambda \ne 0$, $M\in \Matlm$, and $f\in \Fq[X]_1$.
Let  $\hat{M}_1, \dots , \hat{M}_{\hat{n}}$ be fixed representatives in $\Matlm$ of the points of $\Plm (\Fq)$ corresponding to  $\Dtp$. Then $M_1, \dots , M_n$ consist precisely of the zero matrix and the nonzero scalar multiples of $\hat{M}_1, \dots , \hat{M}_{\hat{n}}$. Moreover, $\Fq[X]_1$ can be viewed as the dual of the $\Fq$-vector space $\Matlm$ and $\Cpt$ can be identified with the image of the evaluation map from
$\Fq[X]_1$ into $\Fq^{\hat{n}}$ defined by $f\mapsto  \hat{c}_f:= \big(f(\hat{M}_1), \dots , f(\hat{M}_{\hat{n}})\big)$. Thus the codewords of $C$ %$\Ct$
are obtained from the codewords of $\hC$ %$\Cpt$
by replacing each coordinate of the latter by all its nonzero scalar multiples, and inserting $0$ as an additional coordinate corresponding to the zero matrix. The resulting map $c_f \mapsto \hat{c}_f$ is a linear projection $\pi$ of $C$ onto $\hC$ %$\Ct$ onto $\Cpt$
with the property that $\wH( c) = (q-1)\wH (\pi(c))$ for all $c\in C$ %\Ct$
and
$ \Vert \pi^{-1} (\hat{D}) \Vert = (q-1) \Vert \hat{D} \Vert $ for all %$c\in C$ and
$\hat{D}\subseteq \hC$. %\Cpt$.
Moreover,  $\{\pi^{-1}(\hat{D}) : \hat{D} \text{ a  $r$-dimensional subcode of }\hC\}$ is precisely the family %set
of all $r$-dimensional subcodes of $C$.  This yields the desired result.
\end{proof}

As indicated by the above proof, the code $\Ct$ is degenerate. %whereas
However, $\Cpt$ is nondegenerate. %Moreover,
The length and dimension of these two codes are easily obtained. The former goes back at least to Landsberg \cite{Land} who obtained a formula for $n$, or rather the number of matrices in $\Matlm$ of a given rank $t$ in case $q$ is prime. We outline a proof for the sake of completeness.

\begin{proposition}
\label{DimLength}
%The code
$\Cpt$ is nondegenerate of dimension $\hat{k} =\ell m$ and length %$\hat{n}$, where
%$\sum_{j=1}^t \hat{\mu}_j(\ell,m)$, where %$\mu_0(\ell,m):= 1$ and for $1\le j\le \ell$,
%$\mu_j(\ell,m):= \left| \{ M\in \Matlm(\Fq): \rk (M) =j\} \right| $ is given by
$$
\hat{n} = \sum_{j=1}^t \hat{\mu}_j(\ell,m) \quad \text{where} \quad
\hat{\mu}_j(\ell,m)  %:= \left\| \{ M\in \Matlm(\Fq): \rk (M) =j\} \right\|
 := %q^{\frac{j(j-1)}{2}} \frac{q^{\frac{j(j-1)}{2}}}{q-1}
 \frac{q^{\binom{j}{2}}}{q-1} \prod_{i=0}^{j-1} \frac{ \left( q^{\ell-i}-1\right) \left( q^{m-i}-1\right) }{q^{i+1}-1}.
$$
\end{proposition}

\begin{proof}
If $\hat{M}$ is a nonzero matrix in $\Matlm$, then its $(i,j)^{\rm th}$ entry is nonzero for some $i,j$, and if we let $f=X_{ij}$, then $f\in \Fq[X]_1$ and $f(\hat{M}) \ne 0$. This implies that $\Cpt$ is nondegenerate.
Similarly, it is %readily seen
clear that the evaluation map \eqref{EvMap} is injective %linear and thus
and so $\dim \Cpt = \dim \Ct = \ell m$. Also
%The length, say $n$, of $\Ct$ is clearly given by $\sum_{j=0}^t {\mu}_j(\ell,m)$, where for $0\le j\le \ell$
$$
\text{length}(\Ct) = |\Dt| = \sum_{j=0}^t  %{\mu}_j(\ell,m)
|\Ej|, \ \text{where} \  \Ej:= \{ M\in \Matlm: \rk (M) =j\}.
$$
%where ${\mu}_j(\ell,m)$ is the number of $M\in \Matlm(\Fq)$ of rank $j$.
The map that sends $M$ to its row-space is clearly a surjection of
%the space of all $\ell\times m$ matrices of rank $j$
$\Ej$ onto the space $G_{j,m}$ of $j$-dimensional subspaces of $\Fq^m$. Moreover for a given  $W\in  G_{j,m}$, the number of $M\in \Matlm$ with row-space $W$ is the number of $\ell\times j$ matrices over $\Fq$ of rank $j$ or equivalently, the number of $j$-tuples of linearly independent %(column)
vectors in $\Fq^{\ell}$. Since $|G_{j,m}|$ is %given by
the Gaussian binomial coefficient ${{m} \brack{j}}_q$,
we find %it follows that
\begin{equation}
\label{mur}
%{\mu}_j(\ell,m) =
|\Ej| = {{m} \brack{j}}_q \prod_{i=0}^{j-1} (q^{\ell} -q^i) = %q^{\frac{j(j-1)}{2}}
q^{\binom{j}{2}} \prod_{i=0}^{j-1} \frac{ \left( q^{\ell-i}-1\right) \left( q^{m-i}-1\right) }{q^{i+1}-1}.
\end{equation}
%This together with Proposition \ref{AffProj} and the fact that
Since $|\Ez|=1$, in view of  Proposition \ref{AffProj}, we obtain the desired formula for $\hat{n}$.
\end{proof}

\begin{remark}
\label{rem:HilbFn}
{\rm
(i) Considering column-spaces instead of row-spaces in the above proof, we obtain the alternative
formula $|\Ej| = {{\ell} \brack{j}}_q \prod_{i=0}^{j-1} (q^{m} -q^i) $ for $0\le j\le \ell$. %, another %alternative
%formula for the number of $\ell\times m$ matrices over $\Fq$ of a given rank $j$: %, viz., %; namely,
%$$
%%{\mu}_j(\ell,m) =
%|\Ej| = {{\ell} \brack{j}}_q \prod_{i=0}^{j-1} (q^{m} -q^i)  %q^{\frac{j(j-1)}{2}}
%\text{where, as usual,}
%%q^{\binom{j}{2}} \prod_{i=0}^{j-1} \frac{ \left( q^{\ell-i}-1\right) \left( q^{m-i}-1\right) }{q^{i+1}-1}.
%$$

(ii) An alternative, albeit rather contrived, way to prove the nondegeneracy of %the code
$\Cpt$ and to determine its dimension is to establish that the natural embedding $\Dtp \hookrightarrow \PP^{\ell m -1}$ is nondegenerate (i.e., $\Dtp$ is not contained in a hyperplane of $\PP^{\ell m -1}$). To this end, one may observe that if $\EuX$ is a projective algebraic variety in $\PP^{k-1}$ (over a field $\F$ say) defined by a homogeneous ideal $I$ of the polynomial ring $S$ in $k$ variables over $\F$, then the least $k'\le k$ such that $\EuX$ is nondenerate in $\PP^{k'-1}$ is given by the Hilbert function of $\EuX$ evaluated at $1$, i.e., by $\dim S_1/I_1$. In case $\EuX = \Dt$, the Hilbert function is known as a consequence of the straightening law of Doubilet-Rota-Stein or an explicit formula due to Abhyankar (see, e.g., \cite{survey}). One sees, in particular that its value at $1$ is the number of ``standard Young bitableaux of area $1$ and bounded by $(\ell \mid m)$'', which is clearly equal to $\ell m$.
}
\end{remark}
Determining the minimum distance of $\Cpt$ isn't quite obvious. To get some feel for this, let us work out some simple examples and also observe that a bound can be readily obtained in a special case.

\begin{example}
{\rm
(i) If %$\ell=1$ or if $\ell=m= t$,
$t=\ell = \min\{\ell , m\}$, then $\Dtp = \PP^{\ell m-1}$ and $\Cpt$ is a first order projective Reed-Muller code (cf. \cite{L}), and in fact, a simplex code. %It is readily seen that
Evidently, it has length ${(q^{\ell m} -1)}/{(q-1)}$ and minimum distance $q^{\ell m -1}$.
%its length is ${(q^{\ell m} -1)}/{(q-1)}$ and the minimum distance ia $q^{\ell m -1}$.
%the length $\hat{n}$ and the minimum distance $\hat{d}$ of $\Cpt$ is given~by
%$$
%\hat{n} = \frac{q^{\ell m} -1}{q-1} \quad \text{and} \quad \hat{k} = q^{\ell m -1}.
%$$

(ii) If $\ell = m = t+1$, then $\Dt = \Matlm\setminus \GL_\ell(\Fq)$ while $\Dtp$ is the hypersurface in $\PP^{\ell^2-1}$ given by $\det (X)=0$. Now the minimum distance, say $\hat{d}$, of the code
${\widehat{C}_{\det}(t;\ell, \ell)}$ corresponding to the projective system $\Dtp\hookrightarrow \PP^{\ell m-1}$ is given by
%$\hat{d} = \hat{n} - \max_H |\Dtp \cap H |$, where $\hat{n} = | \Dtp |$
$$
\hat{d} = \hat{n} - \max_H |\Dtp \cap H |, \quad \text{where} \quad \hat{n} = | \Dtp | =  \frac{q^{\ell^2} -1}{q-1}   - q^{\binom{\ell}{2}} \prod_{i=2}^{\ell} (q^i-1)
$$
denotes the length of ${\widehat{C}_{\det}(t;\ell, \ell)}$ and the maximum is over all hyperplanes $H$ in $\PP^{\ell^2-1}$. The irreducible polynomial $\det(X)$, when restricted to $H$ gives rise to a (possibly reducible) hypersurface in $\PP(H)\simeq \PP^{\ell^2-2}$ of degree $\le \ell$. Hence by Serre's inequality (cf. \cite{Se}; see also \cite{L}),
$$
|\Dtp \cap H | \le \ell q^{\ell^2 -3 } + \frac{q^{\ell^2-3} -1}{q-1}
% \quad \text{which implies} \quad
$$
This yields the following lower bound on the minimum distance of ${\widehat{C}_{\det}(t;\ell, \ell)}$.
$$
\hat{d} \ge q^{\ell^2 -1}  + q^{\ell^2 -2 } - (\ell-1) q^{\ell^2 -3 } - q^{\binom{\ell}{2}} \prod_{i=2}^{\ell} (q^i-1).
$$
In the special case when $\ell=m=2$ and $t=1$, we find $|\Dtp \cap H | \le 2q+ 1$ and $\hat{d} \ge q^{2}$. The Serre bound $2q+1$ is attained if we take $H$ to be any of the coordinate hyperplanes. %It follows that
%$q^2$ is the true
Hence the minimum distance of $\widehat{C}_{\rm det}(1;2,2)$ is $q^2$.
}
\end{example}

\section{Weight Distribution} %A Basic Reduction and the Case of $2\times 2$ minors}
\label{sec:wts}

%The codewords of $\Ct$ as well as $\Cpt$ are parametrized by linear homogeneous polynomials
%$f = \sum_{i,j} f_{ij}X_{ij}$

It turns out that the Hamming weights of codewords of $\Ct$ as well as $\Cpt$ are few in number.

\begin{lemma}
\label{PartialTraceLem}
Let $f(X)=\sum_{i=1}^\ell\sum_{j=1}^m f_{ij}X_{ij}$ be a linear homogeneous polynomial in $\Fq[X]$. Denote by $F=(f_{ij})$ the coefficient matrix of $f$. Then the Hamming weights of the corresponding
codewords $c_f$ of $\Ct$ and $\hat{c}_f$ of $\Cpt $ depend only on $\rk(F)$. In %particular,
fact, if $r= \rk(F)$, then $\wH(c_f) = \wH(c_{\tau_r})$ and  $\wH(\hat{c}_f) = \wH(\hat{c}_{\tau_r})$, where $\tau_r:= X_{11} + \cdots + X_{rr}$ %denotes
is the $r^{\rm th}$ partial trace of $X$.
%Let $c_f$ denote the codeword of $\Ct$ corresponding to a
\end{lemma}

\begin{proof}
We have seen in the proof of Proposition \ref{AffProj} that $\wH (c_f) = (q-1)\wH(\hat{c}_f)$ and thus it suffices to only consider $\wH(c_f)$. If $P\in \GL_{\ell}(\Fq)$ and $Q\in \GL_m(\Fq)$ are nonsingular matrices, then $g(X)= f(P^TXQ^T)$ is in $\Fq[X]_1$ and its coefficient matrix is $PFQ$. Moreover, the
map $M\mapsto P^TMQ^T$ is a rank-preserving bijection of $\Matlm$ onto itself that induces a bijection of
the support of $c_g$ onto the support of $c_f$. Hence $\wH(c_g) =\wH(c_f)$.
%Note that $\wH(c_f) = |\{M\in \Matlm : \rk (M) \le t \text{ and } f(M)\ne 0\}$.
In particular, if $r= \rk(F)$, then we can choose $P$ and $Q$ such that
$$
PFQ=\left[\begin{array}{c|c}I_r & 0\\ \hline 0 & 0\end{array}\right].
$$
Consequently, $\wH(c_f) = \wH(c_{\tau_r})$  and $\wH(\hat{c}_f) = \wH(\hat{c}_{\tau_r})$.
%This complets the proof.
\end{proof}

\begin{corollary}
\label{corwts}
Each of the codes $\Ct$ and $\Cpt$ have at most $\ell+1$ distinct weights, $w_0, w_1, \dots, w_{\ell}$ and
$\hat{w}_0, \hat{w}_1, \dots, \hat{w}_{\ell}$ respectively, given by %$w_0=0=\hat{w}_0$ while
$w_r = \wH(c_{\tau_r})$ and $\hat{w}_r = \wH(\hat{c}_{\tau_r})  = w_r/(q-1)$ for $r=0,1,\dots , \ell$. Moreover, the weight enumerator polynomials $A(Z)$
%number $A_{w_r}$ (resp. $\hat{A}_{\hat{w}_r}$) of codewords
of $\Ct$ and $ \hat{A} (Z)$ of $\Cpt$ are given by
$$
A(Z) = \sum_{r=0}^{\ell}  \mu_r(\ell, m) Z^{w_r} \quad \text{and} \quad \hat{A}(Z) =   \sum_{r=0}^{\ell}  \mu_r(\ell, m) Z^{\hat{w}_r},
$$
%of weight $w_r$ (resp. $\hat{w}_r$)  is given by $A_0=1=\hat{A}_0$ and
%$A_{w_r} = \mu_r(\ell, m):= (q-1) \hat{\mu}_r(\ell,m) = \hat{A}_{\hat{w}_r}$ for $1\le r \le \ell$.
%$A_{w_r} = \mu_r(\ell, m)$ and $\hat{A}_{\hat{w}_r} = \hat{\mu}_r(\ell,m)$ for $1\le r \le \ell$, where $\hat{\mu}_r(\ell,m)$ is as in Proposition \ref{DimLength} and
where $\mu_r(\ell, m)$ %= (q-1)\hat{\mu}_r(\ell, m)$ is
is the number of $\ell\times m$ matrices over $\Fq$ of rank $r$, given by
$$
%\mu_r(\ell, m) =
q^{\binom{r}{2}} \prod_{i=0}^{r-1} \frac{ \left( q^{\ell-i}-1\right) \left( q^{m-i}-1\right) } {q^{i+1}-1} =  {{m} \brack{r}}_q \prod_{i=0}^{r-1} (q^{\ell} -q^i) =  {{\ell} \brack{r}}_q \prod_{i=0}^{r-1} (q^{m} -q^i) .
$$
\end{corollary}

\begin{proof}
%The assertions about the weights are
Since $\ell \le m$, the rank of any $F\in \Matlm$ is at most $\ell$. Thus the desired result
%This is an immediate consequence of
follows from Lemma \ref{PartialTraceLem} together with \eqref{mur} and part (i) of Remark \ref{rem:HilbFn}. %The stated expression
\end{proof}

We remark that it is not clear, {\em a priori}, that the weights $w_r$ are distinct for distinct values of $r$.  Also it isn't clear which of the nonzero weights $w_1, \dots , w_r$ is the least. But the weight distribution or the spectrum is completely determined once we solve the combinatorial problem of counting the number of $\ell \times m$ matrices $M$ over $\Fq$ of rank $\le t$ for which $\tau_r(M)\ne 0$.
As indicated in the Introduction, Delsarte \cite{D} solved an essentially equivalent problem of determining the number $N_t(r)$ of %$\ell \times m$ matrices
$M \in \Matlm(\Fq)$ of rank $t$ with $\tau_r(M)\ne 0$, and showed: % that
$$
%N_t(r) = \frac{ (q-1 )\left( \mu_t - P_t(r) \right)}{q} \quad \text{where} \quad
%P_t(r) = \sum_{i=0}^{\ell} (-1)^{t-i} q^{im + \binom{k-i}{2}} {{m-i}\brack{m-t}}_q {{m-r}\brack{i}}_q.
N_t(r) = \frac{ (q-1 )}{q} \left( \mu_t (\ell,m) -
 \sum_{i=0}^{\ell} (-1)^{t-i} q^{im + \binom{t - i}{2}} {{\ell -i}\brack{\ell -t}}_q {{\ell -r}\brack{i}}_q  \right),
$$
where $\mu_t (\ell,m)$ is as in Corollary~\ref{corwts} above.  Consequently, the nonzero weights of $\Ct$ are given by
$w_r = \sum_{s=1}^t N_s(r)$ for $r=1, \dots , \ell$. However, for a fixed $t$ (even in the simple case $t=1$), it is not entirely obvious how $w_1, \dots , w_{\ell}$ %$N_t(1), \dots , N_t(\ell)$
are ordered and which among them is the least. In the next section, we circumvent these difficulties and use a direct approach in the case $t=1$.

\section{Case of $2\times 2$ minors}

In this section we consider the determinantal variety ${\mathcal D}_1$ defined by the vanishing of all $2\times 2$ minors of $X$, and show that the weight distribution of the corresponding code is explicitly determined in this case. We begin by recalling an elementary and well-known characterization of rank
$1$ matrices as outer (or dyadic) products of nonzero vectors.

\begin{proposition}
\label{Rank1}
Let $\F$ be a field and let $M \in\Matlm(\F)$. % an $\ell \times m$ matrix with entries in $\F$.
Then $\rk (M)=1$ if and only if there are nonzero (row) vectors $\mathbf{u}\in \F^{\ell}$ and $\mathbf{v} \in \F^m$ such that
$M= \mathbf{u}^T\mathbf{v}$. Moreover, if $\mathbf{u}^T\mathbf{v}= \mathbf{a}^T\mathbf{b}$
for  nonzero $\mathbf{u}, \mathbf{a} \in \F^{\ell} $ and $ \mathbf{v}, \mathbf{b}\in \F^{m}$, then
%there is $\lambda \in \F$ with
$ \mathbf{a}  = \lambda  \mathbf{u}$ and $ \mathbf{b} = \lambda^{-1} \mathbf{v}$
for a unique $\lambda \in \F$ with $\lambda\ne 0$.
\end{proposition}

The complete weight distribution of determinantal codes in the case $t=1$ is given by the following theorem together with  Corollary \ref{corwts}. In the statement of the theorem, we restrict to $\Cpt$, but the corresponding result for $\Ct$ when $t=1$ is readily obtained (and is evident from the proof).

\begin{theorem}
\label{Wts}
The %distinct only
nonzero weights of $\Cpone$ are $\hat{w}_1, \dots , \hat{w}_{\ell}$, given by
$$
\hat{w}_r =  \wH(\hat{c}_{\tau_r}) = q^{\ell + m -2} + q^{\ell + m-3} + \dots +  q^{\ell + m -r-1} = q^{\ell + m -r -1} \frac{(q^r-1)}{q-1} % \quad \text{for } r=1, \dots , \ell.
$$
for $r=1, \dots , \ell$.
In particular, $\hat{w}_1 < \hat{w}_2 < \dots  < \hat{w}_{\ell}$
and the minimum distance of $\Cpone$ is $q^{\ell + m -2}$.
\end{theorem}

\begin{proof}
Fix $r\in \{1, \dots , \ell\}$ and let $\tau_r$ be as in Lemma \ref{PartialTraceLem} and $c_{\tau_r}$ the corresponding element of $\Cone$. Write $w_r:= \wH(c_{\tau_r})$.
In view of Proposition \ref{AffProj} and Corollary \ref{corwts}, it suffices to show that %$\wH(c_{\tau_r})
$w_r= q^{\ell + m -r -1} {(q^r-1)}$. % for $r=1, \dots , \ell$.
%Since the only matrix of rank $0$ is the zero matrix, %for $\Cone$ we have
To this end, first observe that
$w_r = \left| \{ M\in \Matlm(\Fq) : \rk(M)=1 \text{ and } \tau_r(M)\ne 0\}\right|$.
By Proposition \ref{Rank1}, every $M\in \Matlm$ of rank $1$ is of the form $\mathbf{u}^T\mathbf{v}$
for some nonzero $\mathbf{u} = (u_1, \dots , u_{\ell})$ and $\mathbf{v}=(v_1, \dots , v_m)$. Now $\tau_r(M) = u_1v_1+ \cdots + u_rv_r$ and for it to be nonzero, we can choose $\mathbf{u}^{(r)}:= (u_1, \dots , u_r)$ in $q^r -1$ ways
(excluding the zero vector) and $(v_1, \dots, v_r)$ to be any element of $\Fq^r$ that is not in the kernel of the restricted inner product map of $\phi: \Fq^r \to \Fq$
%given by $\mathbf{x} \mapsto \mathbf{u}^{(r)} \mathbf{x}^T$;
defined by $\phi(\mathbf{x}):= \mathbf{u}^{(r)} \mathbf{x}^T$. Since $\phi$ is a nonzero linear map, the kernel is a $\Fq$-vector space of dimension $r-1$, and %hence
so $(v_1, \dots, v_r)$ can be chosen in $q^r - q^{r-1}$ ways. Finally, the remaining coordinates
$u_{r+1}, \dots , u_{\ell}$ of $\mathbf{u}$ can be chosen arbitrarily in $q^{\ell -r}$ ways while the remaining coordinates of
$\mathbf{v}$ can be chosen in $q^{m-r}$ ways. Since $\mathbf{u}$ and $\mathbf{v}$ are determined by
$M$ only up to scaling by a nonzero element of $\Fq$, it follows that
$$
w_r = \frac{(q^r-1)\left(q^r - q^{r-1} \right) q^{\ell -r}q^{m -r}}{q-1} = q^{\ell + m -r -1} {(q^r-1)},
$$
as desired.
\end{proof}

\begin{remark}
{\rm
It may be noted that the exponent $\ell + m -2$ of $q$ in the minimum distance $\Cpone$ is precisely the dimension of the determinantal variety $\Dtp$ when $t=1$. Also,
specializing Proposition \ref{DimLength} to $t=1$ we see that the length $n$ of $\Cpone$ is
$(q^{\ell}-1)(q^m-1)/(q-1)^2$, which is a monic polynomial in $q$ of degree $\ell+m-2$.
%Since the minimum distance is $q^{\ell+m-2}$, the result follows.
It follows that the relative distance $\delta = d/n$ of $\Cpone$ is asymptotically equal to $1$ as $q\to \infty$. On the other hand, the rate $R=k/n$ is quite small as $q\to \infty$, but it tends to $1$ as $q\to 1$.
}
\end{remark}

%%% Older Corollary - Now downgraded to a remark %%%%%%%
%\begin{corollary}
%\label{cor:reldist}
%The relative distance $\delta = d/n$ of $\Cpone$ is asymptotically equal to $1$ as $q\to \infty$.
%\end{corollary}
%
%\begin{proof}
%Specializing Proposition \ref{DimLength} to $t=1$ we see that the length $n$ of $\Cpone$ is
%$(q^{\ell}-1)(q^m-1)/(q-1)^2$, which is a monic polynomial in $q$ of degree $\ell+m-2$. Since the minimum distance is $q^{\ell+m-2}$, the result follows.
%\end{proof}

We now turn to the determination of the higher weights or the generalized Hamming weights of $\Cpone$.
As remarked in the Introduction, the first $\ell$ higher weights $\hat{d}_1, \dots , \hat{d}_{\ell}$ coincide with the nonzero weights $\hat{w}_1, \dots , \hat{w}_{\ell}$ given by Theorem \ref{Wts}. Of course there are many more higher weights, namely, $\hat{d}_1, \dots , \hat{d}_k$, where $k=\ell m$, that are to be determined. It turns out that it is easy to find the first $m$ of them and also to show that these meet the Griesmer-Wei bound.

\begin{theorem}
\label{HigherWts}
For $ r=1, \dots , m$, the $r^{\rm th}$ higher weight $\hat{d}_r$ of $\Cpone$ meets the Griesmer-Wei bound and is given by
$$
\hat{d}_r = q^{\ell + m -2} + q^{\ell + m-3} + \dots +  q^{\ell + m -r-1} = q^{\ell + m -r -1} \frac{(q^r-1)}{q-1} .
$$
In particular, if $r\le \ell$ and $\hat{w}_r$ is as in  Theorem \ref{Wts}, then $\hat{d}_r =\hat{w}_r$.
\end{theorem}

\begin{proof} Fix $r\in \{1, \dots , m\}$ and let $L_r$ be the $r$-dimensional subspace of $\Fq[X]_1$ generated by $X_{11}, \dots , X_{1r}$. Also let $D_r = \Ev (L_r)$ be the corresponding subcode of $\Cpone$. Since $\Ev$ is injective linear, $\dim D_r = r$. Moreover, since the coefficient matrix of
any $f\in L_r$ has all rows except the first consist entirely of zeros, it follows from Lemma \ref{PartialTraceLem} and Theorem \ref{PartialTraceLem} that $\wH(c) = \hat{w}_1 =  q^{\ell + m -2}$ for all nonzero $c\in D_r$. We now use a  well-known formula for the support weight of an $r$-dimensional subcode (see, e.g., \cite[Lemma 12]{GPP}) to obtain
$$
 \| D_r \| = \frac{1}{q^r - q^{r-1}} \sum_{c\in D_r} \wH(c) = \frac{1}{q^r - q^{r-1}}  (q^r - 1) q^{\ell + m -2} = q^{\ell + m -r -1} \frac{(q^r-1)}{q-1}.
$$
%Thus
%So $\hat{d}_r \le \hat{w}_r$.
On the other hand,  the Griesmer-Wei bound \cite[Corollary 3.3]{TV2} gives
$$
\hat{d}_r \ge \sum_{j=0}^{r-1} \left\lceil \frac{\hat{d}_1}{q^j} \right\rceil = \sum_{j=0}^{r-1} q^{\ell + m -2-j} = q^{\ell + m -r -1} \frac{(q^r-1)}{q-1}. % = \hat{w}_r.
$$
The last two displayed equations imply the desired formula for $\hat{d}_r$ and show that the   Griesmer-Wei bound is met, and also that $\hat{d}_r = \hat{w}_r$  if $r\le \ell$.
\end{proof}
%
%\begin{remark}
%{\rm
%The equality $\lceil \hat{d}_1/q^j\rceil = q^{\ell + m -2-j}$ used in the above proof holds not only for $j\le m-1$, but more generally, for $j \le \ell+m-2$. Thus the Griesmer-Wei bound gives an upper bound
%}
%\end{remark}

We can push the techniques used in the  above proof  to obtain lower and upper bounds for some of the subsequent higher weights, %namely, viz.,
$\hat{d}_r$ for $m< r < \ell + m$. %, and obtain lower and upper bounds for these.

\begin{lemma}
\label{lowerupper}
%For $ r=m+1, \dots , \ell+m-1$, the $r^{\rm th}$ higher weight $\hat{d}_r$ of $\Cpone$ satisfies
Assume that $\ell \ge 2$. Then
for $ s=1, \dots , \ell-1$, the $(m+s)^{\rm th}$ higher weight $\hat{d}_{m+s} $
of $\Cpone$ satisfies
$$
q^{\ell  - s -1} \frac{(q^{m+s}-1)}{q-1} = \hat{d}_m +  q^{\ell  - s -1} \frac{(q^{s}-1)}{q-1}  \le \hat{d}_{m+s} \le
\hat{d}_m +  q^{\ell +m - s -2} \frac{(q^{s}-1)}{q-1} ,
$$
where $\hat{d}_m$ is as in Theorem \ref{HigherWts}. In particular, $\hat{d}_m +  q^{\ell  - 2} \le \hat{d}_{m+1} \le \hat{d}_m +  q^{\ell  +m -3}$.
\end{lemma}

\begin{proof} Fix $s\in \{1, \dots , \ell -1\}$. The lower bound for $ \hat{d}_{m+s}$ is precisely the Griesmer-Wei bound, and is obtained exactly as in the proof of Theorem \ref{HigherWts}. To obtain the upper bound, consider the subspace $L_{m+s}$ of $\Fq[X]_1$ spanned by $X_{11}, \dots , X_{1m}$ and $X_{21}, \dots , X_{s+1 \, 1}$.  %The coefficient matrix $F=\left(f_{ij}\right)$ of a
%Observe that
Any nonzero $f = \sum_{j=1}^m f_{1j} X_{1j} + \sum_{i=2}^{s+1} f_{i1} X_{i1}$ in $L_{m+s}$ is in either of the following three disjoint classes: (i)  $\EuF_1$ consisting of those $f\ne 0$ for which $f_{i1}=0$ for $2\le i \le s+1$, (ii)
$\EuF_2$ consisting of those $f\ne 0$ for which $f_{i1} \ne 0$ for some $i$ with $2\le i \le s+1$ and $f_{1j}=0$ for $2\le j \le m$, and (iii) $\EuF_3$ consisting of all the other $f\ne 0$. Note that %Clearly,
$$
%|\EuF_1| = (q^m-1), \quad |\EuF_2| = q(q^s-1), \quad |\EuF_3| = \left(q^{m+s}-1\right) - (q^m-1) - q(q^s-1).
|\EuF_1| = (q^m-1), \  \, |\EuF_2| = q(q^s-1), \text{ and } \, |\EuF_3| = \left(q^{m+s}-1\right) - (q^m-1) - q(q^s-1).
$$
Also note that the coefficient matrix $F=\left(f_{ij}\right)$ of $f$ has rank $1$ if $f\in \EuF_1 \cup \EuF_2$, and rank $2$   if $f\in \EuF_3$.
Now  let  $D_{m+s} = \Ev\left( L_{m+s}\right)$ be the corresponding %$(m+s)$-dimensional
subcode of $\Cpone$.
%= \Ev\left( L_{m+s}\right)$,
Then $\dim D_{m+s} = m+s$ and % from
by Lemma \ref{PartialTraceLem} and Theorem \ref{Wts}, %we find % see that
\begin{eqnarray*}
 \| D_{m+s} \| &=& \frac{1}{q^{m+s} - q^{m+s-1}} \sum_{ c\in D_{m+s}} \wH(c)  \\
&=& \frac{q^{-(m+s-1)}}{q-1} \left[ \left(q^m-1 \right) q^{\ell + m -2} + \left(q^{s+1}-q \right) q^{\ell + m -2}    \right. \\
& & \qquad \qquad \qquad + \left. \left( q^{m+s} - q^m - q^{s+1} + q \right) \left( q^{\ell + m -2}  +  q^{\ell + m -3}\right) \right] \\
&=& q^{\ell -1} \frac{(q^{m}-1)}{q-1} +  q^{\ell +m - s -2} \frac{(q^{s}-1)}{q-1} .
\end{eqnarray*}
Thus, in view of Theorem \ref{HigherWts}, we obtain the desired upper bound for $\hat{d}_{m+s} $.
\end{proof}

It appears interesting to know whether the  higher weights subsequent to $\hat{d}_{m}$
%, \dots , \hat{d}_{m+\ell-1}$
meet the Griesmer-Wei bound. We will show
in Theorem \ref{furtherdr} below
that this is not the case and, in fact, the exact value of $\hat{d}_{m+1}$ is given by the upper bound in the above lemma.  Some spadework is, however, needed.
% and this will be done first.
First, we make an elementary, but useful observation about sums of rank $1$ matrices.

\begin{lemma}
\label{rank1sums}
Let $\F$ be a field and let $\mathbf{u}, \mathbf{a}, \mathbf{x} \in \F^{\ell}$ and
$\mathbf{v}, \mathbf{b}, \mathbf{y}\in \F^m$ be nonzero vectors such that $\mathbf{u}^T\mathbf{v} + \mathbf{a}^T\mathbf{b} = \mathbf{x}^T\mathbf{y}$. %Then %either
Denote by $\langle\mathbf{u}, \mathbf{a}, \mathbf{x}\rangle$ the subspace of $\F^{\ell}$ spanned by
$\mathbf{u}, \mathbf{a}, \mathbf{x}$, %is one-dimensional,
and by
$\langle\mathbf{v}, \mathbf{b}, \mathbf{y}\rangle$ the subspace of $\F^{m}$ spanned by
$\mathbf{v}, \mathbf{b}, \mathbf{y}$. Then $\langle\mathbf{u}, \mathbf{a}, \mathbf{x}\rangle$
is one-dimensional or
$\langle\mathbf{v}, \mathbf{b}, \mathbf{y}\rangle$
 is one-dimensional.
\end{lemma}

\begin{proof}
In view of the uniqueness up to multiplication by nonzero scalars of vectors in an outer product (cf. the last assertion in Proposition \ref{Rank1}), it suffices to show that
%$\mathbf{u}, \mathbf{a}$ are linearly dependent or $\mathbf{v}, \mathbf{b}$ are linearly dependent.
at least one among $\{\mathbf{u}, \mathbf{a}\}$ and $\{\mathbf{v}, \mathbf{b}\}$ is linearly dependent.
 Suppose  $\{\mathbf{u}, \mathbf{a}\}$ is linearly independent. Then $u_ia_j - u_j a_i\ne 0$ for some $i,j\in \{1, \dots, \ell\}$ with $i\ne j$. Here, as usual,   $u_1, \dots , u_{\ell}$ denote the coordinates of
$\mathbf{u}$  (and likewise for the other vectors). Now choose any $r,s\in\{1, \dots ,m\}$ and consider the
$2\times 2$ minor $\Delta$ of $\mathbf{u}^T\mathbf{v} + \mathbf{a}^T\mathbf{b}$ corresponding to rows indexed by $i,j$ and columns indexed by $r,s$. %Using the multilinearity of the determinant and an
An elementary calculation shows that %, we see that
$$
\Delta = \Delta_1 \Delta_2 \quad \text{where} \quad \Delta_1 :=
\left| \begin{array}{ll} u_i & a_i \\ u_j & a_j \end{array} \right| \text{ and }
\Delta_2 : = \left| \begin{array}{ll} v_r & v_s \\ b_r & b_s \end{array} \right|.
$$
Now $\Delta=0$, being the $2\times 2$ minor of the rank $1$ matrix $\mathbf{x}^T\mathbf{y}$ and $\Delta_1\ne 0$. %Hence
So $\Delta_2=0$. Since $r,s$ %\in\{1, \dots ,m\}$
were arbitrary, we see that
$\{\mathbf{v}, \mathbf{b}\}$ is linearly dependent.
\end{proof}

\begin{corollary}
\label{cor:maxrank1}
Let $\F$ be a field and let $\EuE$ be a subspace of $\Matlm(\F)$
%the $\F$-vector space of all $\ell \times m$ matrices $M$  with entries in $\F$
such that $\rk(M)=1$ for all nonzero $M\in \EuE$. Then
%$\EuE = \left\{ \mathbf{u}^T\mathbf{v} :\mathbf{v} \in V\right\}$ for some $ \mathbf{u}\in \F^{\ell}$ and a subspace $V$ of $\F^m$ or $\EuE = \left\{ \mathbf{u}^T\mathbf{v} :\mathbf{u} \in U\right\}$ for some $ \mathbf{v}\in \F^{m}$ and a subspace $U$ of $ \F^{\ell}$.
\begin{equation}
\label{type1}
\EuE = \left\{ \mathbf{u}^T\mathbf{v} :\mathbf{v} \in V\right\} \quad \text{for some $ \mathbf{u}\in \F^{\ell}$ and a subspace $V$ of $\F^m$}
\end{equation}
or
\begin{equation}
\label{type2}
\EuE = \left\{ \mathbf{u}^T\mathbf{v} :\mathbf{u} \in U\right\} \quad \text{for some $ \mathbf{v}\in \F^{m}$ and a subspace $U$ of $ \F^{\ell}$}.
\end{equation}
In particular, $\dim \EuE \le \max\{\ell, m\} = m$.
%is at most $\max\{\ell, m\}$, i.e., $\dim \EuE \le m$.
\end{corollary}

\begin{proof}
If  $\dim \EuE =0$, then there is nothing to prove. Assume that there is $M\in  \EuE$ with $\rk(M)=1$.
%By Proposition \ref{Rank1}, we can write $M= \mathbf{u}^T\mathbf{v}$ for
Fix some nonzero $\mathbf{u} \in \F^{\ell}$ and $\mathbf{v}\in \F^m$ such that $M= \mathbf{u}^T\mathbf{v}$.
In case every element of $\EuE$ is of the form $\mathbf{u}^T\mathbf{b}$
for some $\mathbf{b}\in \F^m$, then $\EuE$ has a basis of the form $\left\{\mathbf{u}^T\mathbf{v}^{(1)}, \dots , \mathbf{u}^T\mathbf{v}^{(s)}\right\}$. In this case, $\left\{\mathbf{v}^{(1)},\dots , \mathbf{v}^{(s)}\right\}$ must be a linearly independent subset of $\F^m$, and hence \eqref{type1} holds with $\dim \EuE = s \le m$. Likewise, if every element of $\EuE$ is of the form $\mathbf{a}^T\mathbf{v}$, then \eqref{type2} holds with $\dim \EuE \le \ell$. The only remaining case is when $\EuE$ contains an element $N$ of the form $\mathbf{a}^T\mathbf{b}$ for some
$\mathbf{a}\in \F^{\ell}$ and $\mathbf{b}\in \F^m$ such that both $\{\mathbf{u}, \mathbf{a}\}$  and $\{\mathbf{v} ,\mathbf{b}\}$ are linearly independent. But then $M+N$ is nonzero and %Now
by Lemma~\ref{rank1sums},
$\rk(M+N) \ne 1$, which  contradicts the hypothesis on $\EuE$.
%$M+N$ can not have rank $1$. This contradicts the hypothesis on $\EuE$.
\end{proof}

\begin{lemma}
\label{maxrank1s}
%Let $\F$ be a field and let $\EuD$ be a subspace of $\Matlm(\F)$ %with $\dim \EuD = m+s$
%of dimension $m+s$ with $1\le s \le \ell -1$. Then $\EuD$ contains at most $q^{m+s-1}+q^2 - q - 1$ matrices of rank $1$. Consequently, %$\EuD$ contains
%it has at least  $\left(q^{m+s-1} - q \right)(q-1)$ matrices of rank $\ge 2$.
%% Changing m+s to r %%%%%%%%%%
%Assume that $\ell \ge 2$.
Let  $\EuD$ be an $r$-dimensional subspace of $\, \Matlm(\Fq)$  with $r >m$. %m < r < \ell + m$.
Then $\EuD$ contains at most $q^{r-1}+q^2 - q - 1$ matrices of rank $1$. Consequently, $\EuD$
%contains it
has
at least  $\left(q^{r-1} - q \right)(q-1)$ matrices of rank $\ge 2$.
\end{lemma}

\begin{proof}
Let $\EuE$ be a subspace of $\EuD$ of maximum possible dimension with the property that $\rk(M)=1$ for all nonzero $M\in \EuE$. Let $s:=\dim \EuE$. If $s=0$, then $\EuD$ has no matrices of rank $1$ and the result holds trivially. Suppose $s\ge 1$. Now $s \le m < r$ and  \eqref{type1} or  \eqref{type2} holds, thanks to Corollary~\ref{cor:maxrank1}. Assume that  \eqref{type1} holds and fix $ \mathbf{u}$ and $V$ as in \eqref{type1}. If a coset of $\EuE$ in $\EuD$ other than $\EuE$ contains a matrix, say $M$, of rank $1$, then the coset is of the form $M+\EuE$.
%Write $M=\mathbf{a}^T\mathbf{b}$ for some
Fix nonzero $\mathbf{a}\in \F^{\ell}$ and $\mathbf{b}\in \F^m$ such that $M=\mathbf{a}^T\mathbf{b}$.
Then $\{\mathbf{u}, \mathbf{a}\}$ must be linearly independent because otherwise every nonzero matrix in $\EuE+\Fq M$ is of rank $1$, contradicting the maximality of $s$. Further if $N = \mathbf{u}^T\mathbf{v} \in \EuE$ (where
$\mathbf{v} \in V$) is such that $M+N$ has rank $1$, then Lemma \ref{rank1sums} implies that $\mathbf{v} = \lambda \mathbf{b}$ for some $\lambda \in \Fq$. Since $\lambda$ can take $q$ values, it follows that
%every coset of $\EuE$ in $\EuD$ other than $\EuE$
the coset $M+\EuE$ contains at most $q$ matrices of rank~$1$.
Likewise if \eqref{type2} holds and we fix $ \mathbf{v}$ and $U$ as in \eqref{type2}, then we see that
$\{\mathbf{v}, \mathbf{b}\}$ is linearly independent, and the coset $M+\EuE$ contains %has
at most $q$ matrices of rank $1$. In the special case when $s=1$ (which is the only case when both \eqref{type1} and \eqref{type2} hold), we see from Lemma \ref{rank1sums} that $M$ is the only matrix of rank $1$ in the
coset $M+\EuE$. Thus if $s=1$, then there are at most $B^*_1:=(q^{r-1}-1)+ (q-1)$ matrices of rank $1$ in $\EuD$.
And if $s\ge 2$, then  there are at most $B_s:= q(q^{r - s}-1)+ (q^s -1)$ matrices of rank $1$ in $\EuD$. To complete the proof, %it suffices to
observe that
$$
B_2 - B^*_1 = (q-1)^2 > 0 \ \text{ and } \ B_2 - B_s = \left(q^s - q^2)(q^{r-s-1} -1 \right) \ge 0 \text { for } 2\le s < r.
$$
So  $\EuD$  has at most $B_2 = q^{r-1}+q^2 - q - 1$ matrices of rank $1$. %Consequently,
%So at least the remaining $(q^r -1) - B_2 = \left(q^{r-1} - q \right)(q-1)$ nonzero matrices have rank $\ge 2$.
And there remain $(q^r -1) - B_2 = \left(q^{r-1} - q \right)(q-1)$ nonzero matrices whose rank must be $\ge 2$.
% and at least these have rank $\ge 2$
\end{proof}

\begin{theorem}
\label{furtherdr}
Assume that $\ell \ge 2$. For $1\le r \le \ell m$, let $\hat{d}_r$  denote the $r^{\rm th}$ higher weight of $\Cpone$. Then for $r=m+1, \dots , \ell m$, % satisfies
$$
\hat{d}_r \ge  q^{\ell +m   - r -1}\left( \frac{q^{r}-1}{q-1} + q^{r-2} -1 \right) = \hat{d}_m +  q^{\ell +m - r -1} \left( \frac{q^{r-m}-1}{q-1}  + q^{r-2} -1 \right),
$$
%where $\hat{d}_m$ is as in Theorem \ref{HigherWts}.
Moreover, equality holds when $r=m+1$ so that
$\hat{d}_{m+1} = \hat{d}_m +  q^{\ell  +m -3}$.
\end{theorem}

\begin{proof} Fix  $r\in \{m+1, \dots , \ell m\}$ and let $D$ be any $r$-dimensional subcode of $\Cpone$.
Since the evaluation map from $\Fq[X]_1$ onto $\Cpone$ is injective and since polynomials
$f = \sum_{i=1}^\ell\sum_{j=1}^m f_{ij}X_{ij}$ in $\Fq[X]_1$ can be identified with their coefficient matrices $\left( f_{ij}\right)$, we see that $D$ can be identified with an $r$-dimensional subspace $\EuD$ of $\Matlm (\Fq)$. Moreover, the Hamming weight of a codeword in $D$ is completely determined by the rank of the corresponding matrix in $\EuD$, thanks to Lemma \ref{PartialTraceLem}. Thus using Theorem  \ref{Wts}, we see that
$$
 \| D \| = \frac{1}{q^r - q^{r-1}} \sum_{c\in D} \wH(c) \ge \frac{\rho \, q^{\ell + m -2} + (q^r - 1 - \rho) \left(q^{\ell + m -2} + q^{\ell + m -3}\right)}{q^r - q^{r-1}}  ,
$$
where $\rho$ denotes the number of matrices in $\EuD$ of rank $1$. Now we know from Lemma \ref{maxrank1s} that $(q^r - 1 - \rho) \ge \left(q^{r-1} - q \right)(q-1)$ and so
\begin{eqnarray*}
 \| D \| & \ge &   \frac{(q^r - 1) q^{\ell + m -2} +  \left(q^{r-1} - q \right)(q-1) q^{\ell + m -3}}{q^r - q^{r-1}} \\
& = & q^{\ell +m   - r -1}\left( \frac{q^{r}-1}{q-1} + q^{r-2} -1 \right).
\end{eqnarray*}
Since $D$ was an arbitrary $r$-dimensional subcode of $\Cpone$, we obtain the desired inequality for $\hat{d}_r$. The alternative expression in terms of $\hat{d}_m$ is an easy consequence of Theorem \ref{HigherWts}. Finally, if $r=m+1$, then the inequality specializes to $\hat{d}_{m+1} \ge \hat{d}_m +  q^{\ell  +m -3}$. %and using Lemma \ref{lowerupper} we see that the equality holds.
Thus Lemma \ref{lowerupper} yields $\hat{d}_{m+1} = \hat{d}_m +  q^{\ell  +m -3}$.
\end{proof}

It may be interesting to determine the exact value of \emph{all} the higher weights of not only $\Cpone$ but also $\Cpt$ for $1\le t\le \ell$.
%\section{Generalized Hamming Weights}

\section*{Acknowledgments}

The first and the second named authors are grateful to the Indian Institute of Technology Bombay and the % IIT Bombay and DTU Lyngby 
Technical University of Denmark respectively for the warm hospitality and support of short visits to these institutions where some of this work was done. 
We are also grateful to Fernando Pi\~nero for bringing  \cite{C} and \cite{D} to our attention.

\end{document}